\documentclass[12pt]{amsart}

\usepackage{amsmath, amssymb}

\newcommand{\la}{\lambda}
\newcommand{\mcc}{{\raise .45ex \hbox{c}}}
\newcommand{\tsym}{$\mathbb{T}^2$-symmetric}
\newcommand{\etsym}{essentially $\mathbb{T}^2$-symmetric}
\newcommand{\undist}{\overline{\mathbb{D}^2}\setminus \mathbb{T}^2}
\newcommand{\cbidisk}{\overline{\mathbb{D}^2}}
\newcommand{\ip}[2]{\langle #1 , #2 \rangle}

\theoremstyle{plain}
\newtheorem{theorem}[equation]{Theorem}
\newtheorem{corollary}[equation]{Corollary}
\newtheorem{prop}[equation]{Proposition}
\newtheorem*{prop*}{Proposition}
\newtheorem{lemma}[equation]{Lemma}

\theoremstyle{definition}
\newtheorem{definition}[equation]{Definition}
\newtheorem{remark}[equation]{Remark}

\newtheorem{claim}[equation]{Claim}

\numberwithin{equation}{section}

\title{Polynomials defining distinguished varieties}

\author{Greg Knese}
\address{University of California, Irvine \\
         Irvine, CA 92614-3875}
\urladdr{http://www.math.uci.edu/~gknese}
\email{gknese@uci.edu}

\date{\today}

\begin{document}
\bibliographystyle{plain}

\begin{abstract}  Using a sums of squares formula for two variable
  polynomials with no zeros on the bidisk, we are able to give a new
  proof of a representation for distinguished varieties.  For
  distinguished varieties with no singularities on the two-torus, we
  are able to provide extra details about the representation formula
  and use this to prove a bounded extension theorem.
\end{abstract}

\maketitle
\section{Introduction}

Let $\mathbb{D}$ be the unit disk, $\mathbb{T}$ be the unit circle,
and $\mathbb{E}$ be the set $\mathbb{C}\setminus
\overline{\mathbb{D}}$ in $\mathbb{C}$.  Let $\mathbb{D}^2 =
\mathbb{D} \times \mathbb{D}$ be the unit bidisk in $\mathbb{C}^2$.

Broadly speaking, this paper continues the study of plane algebraic
curves (algebraic varieties in $\mathbb{C}^2$) and their interaction
with the two dimensional torus $\mathbb{T}^2 = \mathbb{T} \times
\mathbb{T}$.  When viewing curves in this 
light, one is presented with
a number of interesting classes of curves.  For example, curves
that do not intersect the closed bidisk are described by \emph{stable}
polynomials and are the subject of Geronimo-Woerdeman \cite{GW04},
\cite{GW06} and Knese \cite{gK08}.  Also, curves $V$ for which $V\cap
\mathbb{T}^2$ is a 
determining set for holomorphic functions on $V$
are called \emph{toral} and are the subject of Agler-M\mcc
Carthy-Stankus \cite{AMS06}.  Both types of curves are intimately
related to inner functions and Pick interpolation problems on the
bidisk.  (Stable polynomials will, in fact, play an important role in
this paper later on.)  

In this paper we continue the study of a third type of curve, the
\emph{distinguished varieties}.

\begin{definition} A non-empty set $V$ in $\mathbb{C}^2$ is a
  \emph{distinguished variety} if there is a polynomial $p \in
  \mathbb{C}[z,w]$ such that
\[
V = \{(z,w) \in \mathbb{D}^2: p(z,w) = 0\}
\]
and such that $V$ exits the bidisk through the distinguished boundary:
\[
\overline{V}\cap \partial (\mathbb{D}^2) = \overline{V}\cap (\partial
\mathbb{D})^2
\]
Here we are taking the closure of $V$ in $\overline{\mathbb{D}^2}$.
\end{definition}

A simple example is the variety $\{(z,w) \in \mathbb{D}^2:
z^3=w^2\}$.

Distinguished varieties were defined in Agler-M\mcc Carthy\cite{AM05}
(although they essentially have appeared as far back as W. Rudin
\cite{wR69}).  They have also appeared in the paper by P. Vegulla
\cite{pV08} and Agler-M\mcc Carthy \cite{AM06}.  Aside from their
aesthetic appeal (which we hope to illustrate in this article), here
are a few reasons why we think distinguished varieties are
interesting.

First, every bounded planar domain with finitely many real analytic
boundary curves is biholomorphic to a distinguished variety (this is
proven in Vegulla \cite{pV08}, relying on a result of Fedorov
\cite{sF91} on the existence of ``unramified separating pairs of inner
function'').  Hence, studying distinguished varieties provides a
unified way of studying function theory on planar domains.  

Second, distinguished varieties include many examples of complex
spaces with singularities which have been relatively unstudied in the
realm of spaces of analytic functions (and bounded analytic functions
more specifically).  In many cases, the function theory on a
distinguished variety with singularities is isomorphic to function
theory on the disk (or some other domain) with some type of constraint
imposed (see Agler-M\mcc Carthy \cite{AM07} for more on this).  For
example, the paper \cite{DPRS08} studies Nevanlinna-Pick interpolation
for bounded analytic functions on the disk with the added condition
$f'(0)=0$.  This class of functions is isomorphic to the bounded
analytic functions on the distinguished variety $\{(z,w)\in
\mathbb{D}^2: z^3 = w^2\}$.

Third, distinguished varieties are important in two variable matrix
theory and in Nevanlinna-Pick interpolation on the bidisk.  Indeed,
that is the motivation of the original article Agler-M\mcc Carthy
\cite{AM05}, and we refer the reader to that paper for more details.

In this paper we give a new proof of a representation formula for
distinguished varieties (Theorem \ref{repthm} below) proved in
Agler-M\mcc Carthy \cite{AM05}.  Our new proof is more elementary than
the original proof in that it involves studying polynomials that
define distinguished varieties directly, as opposed to the approach of
the original proof which consisted of constructing certain probability
measures on the boundary of the distinguished variety and studying the
resulting function theory.

More significantly perhaps, we construct a family of ``sums of
squares'' formulas for polynomials defining distinguished varieties
and related polynomials (see Theorems \ref{symstablethm} and
\ref{sosdist}). These extend known sums of squares formulas for stable
polynomials (as in \cite{GW04} and \cite{gK08}) and polynomials with
no zeros on $\mathbb{D}^2$ (as in \cite{CW99} and \cite{gK08}). (See
also Theorems \ref{stablethm} and \ref{GWthm} below.) This approach
allows us to get more detailed information about the representation
formula for distinguished varieties when the variety in question has
no singularities on $\mathbb{T}^2$ (see Theorem \ref{refinerep}).

In turn, this allows us to prove a novel bounded extension theorem
(with estimates) for distinguished varieties with no singularities on
$\mathbb{T}^2$ (see Theorem \ref{extendthm}).  Namely, a polynomial
$f$ of two variables on $V$ can be extended to a rational function $F$
whose supremum norm on $\mathbb{D}^2$ is bounded by a constant times
the supremum norm of $f$ on $V$.  The literature is scattered with a
number of results of this type; here we mention a few.  The paper
\cite{eS75} by Stout, gives necessary and sufficient conditions for
when such an extension can be performed on discs properly embedded in
the polydisc, although due to its more abstract approach does not
provide details about constants.  The paper Adachi-Andersson-Cho
\cite{AAC99} proves extension theorems for analytic subvarieties of
analytic polyhedra using integral formulas.  We believe our extension
theorem, due to its concrete algebraic approach, complements these
other results.

\section{Statements of results} 
\label{results} 

We say a two variable polynomial $p \in \mathbb{C}[z,w]$ has degree
$(n,m)$ if it has degree $n$ in $z$ and $m$ in $w$; we say $p$ has
degree \emph{at most} $(n,m)$ if it has degree at most $n$ in $z$ and
degree at most $m$ in $w$.  Recall that a rational matrix valued
function $\Phi: \mathbb{D} \to \mathbb{C}^{m\times m}$ on the disk is
\emph{inner} if $\Phi$ is unitary-valued on the unit circle.

The main theorem is the following; see Section \ref{mainthmsec} for
our new proof.

\begin{theorem}[Agler-M\mcc Carthy \cite{AM05}] \label{repthm} Let $V$
  be a distinguished variety, defined as the zero set of a polynomial
  $p \in \mathbb{C}[z,w]$ of minimal degree $(n,m)$.  Then, there is an $(m+n)\times
  (m+n)$ unitary matrix $U$ which we write in block form as
\[
U = \begin{matrix} & \begin{matrix} \mathbb{C}^m & \mathbb{C}^n
  \end{matrix} \\
\begin{matrix} \mathbb{C}^m \\ \mathbb{C}^n \end{matrix} &
\begin{pmatrix} A & B \\ C & D \end{pmatrix}
\end{matrix}
\]
such that
\begin{itemize}
\item $D$ has no unimodular eigenvalues,
\item $p(z,w)$ is a constant multiple of
\[
\det \begin{pmatrix} A - wI_m & zB \\ C & zD - I_n \end{pmatrix},
\text{ and}
\]

\item defining the following rational matrix valued inner function:
\[
\Phi(z) = A + zB(I_n - zD)^{-1} C,
\]
we have
\[
V = \{(z,w) \in \mathbb{D}^2: \det(wI_m - \Phi(z))=0\}.
\]
\end{itemize}

\end{theorem}

Naturally, the roles of $z$ and $w$ can be reversed above to give
similar statements.  The original article \cite{AM05} also proves a
converse of the above; namely, if $\Phi$ is a matrix valued rational
inner function on $\mathbb{D}$ then
\[
\{(z,w) \in \mathbb{D}^2: \det(wI_m - \Phi(z))\}
\]
is a distinguished variety.  We are able to provide some additional
information about this representation when $V$ has no singularities on
$\mathbb{T}^2$.  This is the content of Theorem \ref{refinerep}.

We now detail the path to our new proof of Theorem \ref{repthm}.

\begin{prop*} When a distinguished variety is extended to all of
  $\mathbb{C}^2$ using a defining polynomial $p$ of minimal degree,
  the resulting variety
\[
V' = \{(z,w) \in \mathbb{C}^2: p(z,w) = 0\} \subset \mathbb{C}^2
\]
satisfies $V' \subset \mathbb{D}^2 \cup \mathbb{T}^2 \cup
\mathbb{E}^2$.
\end{prop*}
This is Proposition \ref{compprop}.  This gives another way of
defining distinguished varieties that is in many ways easier to work
with.  Namely, instead of the original definition (which looks at
subvarieties of $\mathbb{D}^2$), we can think of distinguished
varieties as algebraic subvarieties of $\mathbb{C}^2$ satisfying $V
\subset \mathbb{D}^2\cup \mathbb{T}^2 \cup \mathbb{E}^2$.  Once
Proposition \ref{compprop} is established we shall think of
distinguished varieties in this way and we use the following
terminology.

\begin{definition} \label{pdefines} We say that a polynomial $p$
  \emph{defines a distinguished variety} if
\[
\{(z,w) \in \mathbb{C}^2: p(z,w) = 0\} \subset \mathbb{D}^2 \cup
\mathbb{T}^2 \cup \mathbb{E}^2.
\]
\end{definition}

\begin{definition} Let $q \in \mathbb{C}[z,w]$ be a polynomial of
  degree at most $(n,m)$.  The \emph{reflection at the degree $(n,m)$}
  of $q$ is defined to be the polynomial $\tilde{q}$ given by
\[
\tilde{q}(z,w) := z^n w^m
\overline{q\left(\frac{1}{\bar{z}},\frac{1}{\bar{w}} \right)}
\]
If the degree at which reflection is applied is not obvious from
context we will state the degree of reflection explicitly (see also
Remark \ref{reflectderivatives}).
\end{definition}

Notice that $|q(z,w)| = |\tilde{q}(z,w)|$ for all $(z,w) \in
\mathbb{T}^2$.  

\begin{definition} We say $q \in \mathbb{C}[z,w]$ is \emph{\etsym} if
\[
q(z,w) = c \tilde{q}(z,w)
\]
for some unimodular constant $c$, and $q$ is \emph{\tsym} if
\[
q(z,w) = \tilde{q}(z,w).
\]
\end{definition}

\begin{prop*} A polynomial $p$ defining a distinguished variety must be
  \etsym.  
\end{prop*}

This is Proposition \ref{symprop}. An \etsym\ polynomial can always be
multiplied by a unimodular constant to make it \tsym\ (i.e. if $p=c
\tilde{p}$, replace $p$ with $\sqrt{c} p$ for some choice of
$\sqrt{c}$).  Since we are mostly concerned with zero sets, for
simplicity we will always make this modification unless otherwise
stated.

Because of these facts there is a direct correspondence between
polynomials that define distinguished varieties and
$\mathbb{T}^2$-symmetric polynomials with no zeros on the set
$\overline{\mathbb{D}^2} \setminus \mathbb{T}^2$.
\begin{lemma} \label{corrlemma} If $p$ is a (\tsym) polynomial of degree
  $(n,m)$ defining a distinguished variety, then
\[
q(z,w) := z^n p\left( \frac{1}{z}, w\right)
\]
is a \tsym\ polynomial of degree $(n,m)$ with no zeros $\undist$. 
\end{lemma}

See Section \ref{sec:prelim} for the proof.

It will soon be made apparent why these \tsym\ polynomials are
preferable to polynomials defining distinguished varieties.  We shall
use the following notations
\[
q_z = \frac{\partial q}{\partial z} \text{ and } q_w = \frac{\partial
  q}{\partial w}
\]
\[
Z_q = \{ (z,w) \in \mathbb{C}^2 : q(z,w) = 0\}.
\]

\begin{remark} \label{reflectderivatives} 
When reflecting $q_z$ or
  $q_w$ the reflection is assumed to be performed at the degree that
  would be generically expected.  Namely, if $q$ has degree $(n,m)$,
  then $q_z$ is reflected at the degree $(n-1,m)$:
\[
\widetilde{q_z} (z,w) = z^{n-1} w^{m} \overline{q_z(1/\bar{z}, 1/\bar{w})}
\]
and $q_w$ is reflected at the degree $(n, m-1)$.  In this case, the
reflection of $\widetilde{q_z}$ at the degree $(n,m)$ is $zq_z(z,w)$.
\end{remark}

We mention and prove some identities relating reflection and
differentiation in order to demystify the expressions appearing in the
theorems to follow.

\begin{lemma} \label{refformulas} Let $q \in \mathbb{C}[z,w]$ be a
  polynomial of degree at most $(n,m)$.  Then,
\[
\begin{aligned}
z\frac{\partial \tilde{q}}{\partial z} (z,w) + \widetilde{q_z}(z,w) &= n
\tilde{q}(z,w)\\ w \frac{\partial \tilde{q}}{\partial w} (z,w) +
\widetilde{q_w}(z,w) &= m \tilde{q}(z,w)
\end{aligned}
\]

In particular, if $q$ is $\mathbb{T}^2$-symmetric,
\[
\begin{aligned}
z q_z (z,w) + \widetilde{q_z}(z,w) &= n q(z,w)\\
w q_w (z,w) + \widetilde{q_w}(z,w) &= m q(z,w).
\end{aligned}
\]
\end{lemma}

\begin{proof} This is a calculus exercise.
\end{proof}

\begin{lemma} \label{modlemma} If $q \in \mathbb{C}[z,w]$, of degree
  at most $(n,m)$, is $\mathbb{T}^2$-symmetric, then for all $a,b \in
  \mathbb{R}$,
\[
\begin{aligned}
&(an+bm)^2|q(z,w)|^2 - 2\text{Re}[(azq_z(z,w)+bwq_w(z,w)) (an+bm)
\overline{q(z,w)}] \\
&= |a\widetilde{q_z}(z,w) + b\widetilde{q_w}(z,w)|^2 - |azq_z(z,w) + b
wq_w(z,w)|^2
\end{aligned}
\]
\end{lemma}

\begin{proof} By the previous lemma,
\[
azq_z(z,w) + b w q_w(z,w) + a\widetilde{q_z}(z,w) + b \widetilde{q_w}(z,w) =
(an+bm) q(z,w).
\]
Observe now (omitting the arguments $(z,w)$; i.e. replacing $q(z,w)$
with $q$)
\[
\begin{aligned}
&(an+bm)^2|q|^2 - 2\text{Re}[(azq_z+bwq_w) (an+bm)
\bar{q}] \\
=& |azq_z + b w q_w + a\widetilde{q_z} + b
\widetilde{q_w}|^2 \\
&- 2\text{Re}[(azq_z+bwq_w)(\overline{azq_z + b w q_w + a\widetilde{q_z} + b
\widetilde{q_w}})] \\
=& |a\widetilde{q_z} + b\widetilde{q_w}|^2 - |azq_z + b
wq_w|^2  
\end{aligned}
\]
\end{proof}

\begin{theorem} \label{thm:tildenozeros} If $q \in \mathbb{C}[z,w]$ is
  a \tsym\ polynomial with no zeros on $\mathbb{D}^2$
  (resp. $\undist$), then $\widetilde{q_z}$ and $\widetilde{q_w}$ have
  no zeros on $\mathbb{D}^2$ (resp. $\undist$).  In addition, for all
  $a,b >0$,
\[
a\widetilde{q_z} + b \widetilde{q_w}
\]
has no zeros on $\overline{\mathbb{D}^2} \setminus (Z_{\widetilde{q_z}}
\cap Z_{\widetilde{q_w}})$.
\end{theorem}

See Section \ref{sossec}.

The following corollary (proved in Section \ref{sossec}) is curious
because it can be iterated.

\begin{corollary} \label{symcor4} If $q \in \mathbb{C}[z,w]$ is an
  \etsym\ polynomial of degree $(n,m)$ with no zeros on $\mathbb{D}^2$
  (or $\undist$), then so is
\[
mn q(z,w) - mzq_z(z,w) - n wq_w(z,w).
\]

If $p$ is a polynomial of degree $(n,m)$ defining a distinguished
variety, then so is
\[
mzp_z(z,w) - n w p_w(z,w).
\]
\end{corollary}

Theorem \ref{thm:tildenozeros} allows us to use the following sum of
squares formula for polynomials with \emph{no zeros on the bidisk}
proved in Cole-Wermer \cite{CW99}.  As is articulated in \cite{CW99},
this formula is equivalent to And\^{o}'s inequality in operator theory
and Agler's Pick interpolation theorem on the bidisk.

\begin{theorem}[Cole-Wermer \cite{CW99}] \label{stablethm} Let $q \in
  \mathbb{C}[z,w]$ be a polynomial of degree at most $(n,m)$ with no
  zeros on the bidisk $\mathbb{D}^2$.  Then, there exists a vector
  polynomial
\[
\vec{A}(z,w) = (A_1(z,w), \dots, A_n(z,w))^t \in \mathbb{C}^n
\]
of degree at most $(n-1,m)$ (meaning each component is a polynomial of
degree at most $(n-1,m)$) and a vector polynomial
\[
\vec{B}(z,w) = (B_1(z,w), \dots, B_m(z,w))^t \in \mathbb{C}^m
\]
of degree at most $(n,m-1)$ such that

\begin{align}
  &q(z,w) \overline{q(Z,W)}  - \tilde{q}(z,w) \overline{\tilde{q}(Z,W)} 
  = \label{sosstable}\\
  &(1-z \bar{Z}) \ip{\vec{A}(z,w)}{\vec{A}(Z,W)} +
  (1-w \bar{W}) \ip{\vec{B}(z,w)}{\vec{B}(Z,W)} \nonumber
\end{align}

for $(z,w), (Z,W) \in \mathbb{C}^2$.  

\end{theorem}

Here $\ip{\vec{A}(z,w)}{\vec{A}(Z,W)}=\sum_j A_j(z,w)
\overline{A_j(Z,W)}$ represents the inner product of the two vectors
in $\mathbb{C}^n$, $\vec{A}(z,w)$ and $\vec{A}(Z,W)$; we emphasize
that this is not any type of Hilbert function space inner product
(likewise for $\vec{B}(z,w)$ and $\vec{B}(Z,W)$, elements of
$\mathbb{C}^m$).

More can be said when $q$ in the above theorem has no zeros on the
closed bidisk $\cbidisk$.  Theorem \ref{GWthm} below is properly
attributed to Geronimo-Woerdeman \cite{GW04} (see Proposition 2.3.3),
although we can more readily explain how this follows from the work in
Knese \cite{gK08} and this is done in the Appendix.

\begin{theorem}[Geronimo-Woerdeman] \label{GWthm} If the polynomial
  $q$ in Theorem \ref{stablethm} has no zeros on $\cbidisk$, then
  $\vec{A}$ and $\vec{B}$ may be chosen so that when we write
\[
\begin{aligned}
\vec{A}(z,w) &= A(w) \begin{pmatrix} 1 \\ z \\ \vdots \\
  z^{n-1} \end{pmatrix} \\
\vec{B}(z,w) &= B(z)  \begin{pmatrix} 1 \\ w \\ \vdots \\
  w^{m-1} \end{pmatrix}
\end{aligned}
\]
where $A(w)$ is an $n\times n$ matrix polynomial in $w$ of degree at
most $m$ in each entry and $B(z)$ is an $m\times m$ matrix
polynomial in $z$ of degree $n$ in each entry, we have
\[
\begin{aligned}
A(w) &\text{ is invertible for all } w \in \overline{\mathbb{D}}
\text{ and}\\
z^n\overline{B(1/\bar{z})} &\text{ is invertible for all }
 z \in \overline{\mathbb{D}}.
\end{aligned}
\]
\end{theorem}

There is an explanation for the apparent asymmetry in this theorem,
which would take us too far afield to detail.  

\begin{remark} In a future article, we will extend this theorem to the
  case where $q$ has no zeros on $\undist$ and finitely many zeros on
  $\mathbb{T}^2$, with the modified conclusion that $A(w)$ as above is
  invertible for $w \in \overline{\mathbb{D}}$ except possibly at
  values of $w \in \mathbb{T}$ for which there exists $z \in
  \mathbb{T}$ such that $q(z,w) = 0$ (and similarly for $B(z)$
  above).  
\end{remark}

Here is our sum of squares formula for \tsym\ polynomials with no
zeros on the bidisk.  It is proved in Section \ref{sossec}.

\begin{theorem} \label{symstablethm} Suppose $q \in \mathbb{C}[z,w]$
  is $\mathbb{T}^2$-symmetric polynomial of degree $(n,m)$ with no
  zeros on $\mathbb{D}^2$ and let $a, b \geq 0$, not both zero.  Then,
  there exists a $\mathbb{C}^n$-valued polynomial $\vec{A}$ of degree
  at most $(n-1,m)$ and a $\mathbb{C}^m$-valued polynomial $\vec{B}$
  of degree at most $(n,m-1)$ such that
\[
\begin{aligned}
  (an+bm)|q(z,w)|^2 &- 2\text{Re}( (azq_z(z,w)+bwq_w(z,w)) \overline{q(z,w)})\\
  &= (1-|z|^2) |\vec{A}(z,w)|^2 + (1-|w|^2) |\vec{B}(z,w)|^2
\end{aligned}
\]
Furthermore, if $q$ is a product of distinct irreducible factors, then
$\vec{A}$ and $\vec{B}$ have at most finitely many common
zeros on $Z_q$.  
\end{theorem}

Here $|\vec{A}(z,w)|$ is the length of the vector $\vec{A}(z,w)$; we
emphasize that it is not any type of function space norm.

\begin{remark}
  Formulas of the above type can always be \emph{polarized} to give a
  formula (as in Theorem \ref{stablethm}) that is holomorphic in
  $(z,w)$ and anti-holomorphic in $(Z,W)$.  This is the polarization
  theorem for holomorphic function which (loosely) says that if
  $H(z,\bar{z}) \equiv 0$ then $H(z,Z) \equiv 0$ for any holomorphic
  function $H$ of two variables.
\end{remark}

Using the correspondence between distinguished
varieties and $\mathbb{T}^2$-symmetric polynomials with no zeros on
$\undist$, we can prove the following sum of squares formula for
polynomials defining distinguished varieties. See Section
\ref{sossec}.

\begin{theorem} \label{sosdist} Let $V$ be a distinguished variety
  given as the zero set of a polynomial $p \in \mathbb{C}[z,w]$ of
  degree $(n,m)$. Let $a,b\geq 0$, not both zero.  Then, there exists
  a $\mathbb{C}^n$-valued vector polynomial $\vec{P}$ of degree at
  most $(n-1,m)$ and a $\mathbb{C}^m$-valued vector polynomial
  $\vec{Q}$ of degree at most $(n,m-1)$ such that
\[
\begin{aligned}
  (bm-an) |p(z,w)|^2 &+ 2\text{Re} [(a zp_z(z, w)-b w
    p_w(z,w))\overline{p(z,w)}]\\ & +
  (1-|z|^2) |\vec{P}(z,w)|^2 \\ =&
  (1-|w|^2) |\vec{Q}(z,w)|^2.
\end{aligned}
\]
If $p$ is a product of distinct irreducible factors, then none of the
entries of $\vec{P}$ or $\vec{Q}$ can vanish identically on $V$.  

Moreover, for $(z,w), (Z,W) \in V$
\begin{equation} \label{sosonV}
(1-z\bar{Z}) \ip{\vec{P}(z,w)}{\vec{P}(Z, W)} =
  (1-w\bar{W}) \ip{\vec{Q}(z,w)}{\vec{Q}(Z,W)}
\end{equation}
\end{theorem}

This last formula \eqref{sosonV} is all that is needed to prove the
representation theorem (and indeed part of the approach in \cite{AM05}
is prove a similar formula. However, with our approach we are able to
``fill out'' the formula \eqref{sosonV} to the rest of $\mathbb{C}^2$.

Using Theorem \ref{GWthm}, we can provide additional information about
this representation when $V$ has no singularities on $\mathbb{T}^2$.
Recall that when a plane curve $V \subset \mathbb{C}^2$ is given as
the zero set of a polynomial with distinct irreducible factors, it has
a singularity at a point in $V$ if and only if both partial
derivatives of the defining polynomial vanish at that point.

\begin{theorem} \label{refinerep} With all notations and assumptions
  as in Theorem \ref{repthm}, there is a  $\mathbb{C}^m$-valued
  polynomial $\vec{Q}(z,w)$ of degree at most $(n, m-1)$ with at most
  finitely many zeros on $V$ such that
\[
\Phi(z) \vec{Q}(z,w) = w \vec{Q}(z,w)
\]
for all $(z,w) \in V$. 

In addition, if $V$ has no singularities on $\mathbb{T}^2$, $\vec{Q}$
may be chosen so that when we write
\[
\vec{Q}(z,w) = Q(z) \begin{pmatrix} 1 \\ w \\ \vdots \\
  w^{m-1} \end{pmatrix}
\]
where $Q(z)$ is an $m\times m$ matrix polynomial of degree at most
$n$ in each entry, we have that $Q(z)$ is invertible for all $z \in
\overline{\mathbb{D}}$. In particular, $\vec{Q}(z,w)$ has no zeros in
$\cbidisk$.  

\end{theorem}

As a corollary, we get the following extension theorem.

\begin{theorem} \label{extendthm} Let $V$ be a distinguished variety
  with no singularities on $\mathbb{T}^2$ and let $\Phi$, $Q$, and
  $\vec{Q}$ be as in Theorem \ref{refinerep}.  Then, for any
  polynomial $f \in \mathbb{C}[z,w]$, the rational function
\[
F(z,w) := (1,0, \dots, 0) Q(z)^{-1} f(zI_m, \Phi(z)) \vec{Q}(z,w)
\]
is equal to $f$ on $V\cap\cbidisk$ and we have the estimates
\[
\begin{aligned}
|F(z,w)| &\leq ||Q(z)^{-1}||\ |\vec{Q}(z,w)| \sup_{V\cap\mathbb{D}^2} |f|\\
& \leq \sqrt{m} ||Q(z)^{-1}||\  ||Q(z)|| \sup_{V\cap\mathbb{D}^2} |f| \\
& \leq C \sup_{V\cap\mathbb{D}^2} |f|
\end{aligned}
\]
for all $(z,w) \in \mathbb{D}^2$, where
\[
C = \sqrt{m} \sup_{z\in\mathbb{D}} ||Q(z)^{-1}||\ ||Q(z)||.
\]
Here we are taking operator norm of the matrices $Q(z)$ and
$Q(z)^{-1}$.
\end{theorem}  

Naturally, the roles of $z$ and $w$ can be reversed in the above
theorem, which will sometimes yield a better constant in the
extension.

The above theorem clearly produces bounded analytic extensions from
$V$ to $\mathbb{D}^2$ for functions other than polynomials.  Indeed,
any function for which we can make sense of $f(zI_m, \Phi(z))$ on the
circle $\mathbb{T}$ will work.  We suspect this can be done for all of
$H^{\infty}(V\cap \mathbb{D}^2)$, but leave this for future work.

\section{Example of Theorem \ref{extendthm}}

Let $b(z)$ be a finite Blaschke product and consider the distinguished
variety 
\[
V = \{(z,w) \in \mathbb{D}^2 : w^m - b(z) = 0\}
\]
This can be represented as the zero set of
\[
\det( wI_m - \Phi(z))
\]
where $\Phi$ is the rational matrix valued inner function
\[
\Phi(z) = \begin{pmatrix} \mathbf{0} & I_{m-1} \\ b(z) &
  \mathbf{0}^t \end{pmatrix}.
\]
If we set
\[
\vec{Q}(z,w) = \begin{pmatrix} 1 \\ w \\ \vdots \\
  w^{m-1} \end{pmatrix}
\]
then
\[
\Phi(z) \vec{Q}(z,w) = w \vec{Q}(z,w)
\]
for $(z,w) \in V$.  Using our method of estimation from Theorem
\ref{extendthm} we see that any polynomial $f$ on $V\cap\mathbb{D}^2$
can be extended to $\mathbb{D}^2$ with its norm increased by at most
the factor $\sqrt{m}$.  (In this case $Q(z)$ from Theorem
\ref{extendthm} is constant and equal to $I_m$.)

When $m=2$ and $b(z) = z^3$ (i.e. $V = \{z^3-w^2=0\}$) this improves
Theorem 2.9 in Knese \cite{gK07}.

\section{Preliminaries} 
\label{sec:prelim}

As mentioned earlier, the following proposition allows us to think of
distinguished varieties in more global terms as subvarieties of
$\mathbb{C}^2$ satisfying \eqref{global}.

\begin{prop} \label{compprop} Let $V \subset \mathbb{D}^2$ be a
  distinguished variety defined by a polynomial $p \in
  \mathbb{C}[z,w]$ of minimal degree $(n,m)$.  Then, the extension of
  $V$ to $\mathbb{C}^2$
\[
V' = \{(z,w) \in \mathbb{C}^2: p(z,w) = 0\}
\]
satisfies
\begin{equation}
V' \subset \mathbb{D}^2 \cup \mathbb{T}^2 \cup \mathbb{E}^2 \label{global}
\end{equation}
Conversely, if $p$ is a polynomial satisfying
\[
W := \{(z,w) \in \mathbb{C}^2: p(z,w) = 0\} \subset \mathbb{D}^2\cup
\mathbb{T}^2 \cup \mathbb{E}^2
\]
then $W\cap \mathbb{D}^2$ is a distinguished variety. 
\end{prop}

\begin{proof}
The converse statement is clear, so we shall focus on the main claim.

It suffices to prove the proposition for each irreducible component of
$V$, so we assume $p$ and $V$ are irreducible.  We emphasize that we
are starting from the assumption
\[
\overline{V}\cap \partial (\mathbb{D}^2) = \overline{V} \cap \mathbb{T}^2
\]
and proving 
\[
V' = Z_p \subset \mathbb{D}^2 \cup \mathbb{T}^2 \cup \mathbb{E}^2.
\] 

First, we claim that $p$ can have no zeros on $\mathbb{D}\times
\mathbb{T}$.  For suppose $p(z_0,w_0)=0$ with $|z_0|<1$ and $|w_0|=1$.
There is a positive integer $k$ such that for $z$ near $z_0$,
$p(z,\cdot)$ has a list of zeros $w_1(z), w_2(z), \dots, w_k(z)$
(listed with possible repetitions) satisfying $w_j(z_0)=w_0$ and the
property that symmetric functions of $w_1, \dots, w_k$ are holomorphic
near $z_0$.  (These are all straightforward consequences of the
Weierstrass preparation theorem or other theorems about the local
behavior of algebraic curves.)  

It cannot be the case that all roots $w_1(z), w_2(z), \dots, w_k(z)$
stay within $\mathbb{T}\cup \mathbb{E}$.  One way to see this is to
observe that $\pi(z) = \prod_{j=1}^{k} w_j(z)$ is holomorphic near
$z_0$ and $\pi(z_0) = w_0^k$.  If $\pi$ is nonconstant, it must assume
a value with modulus less than one, and this implies some sequence of
points in $V$ tends to $\mathbb{D}\times \mathbb{T}$, a contradiction.
If $\pi$ is constant (and therefore unimodular), either every $w_j(z)$
is unimodular valued (forcing all roots to be constant as a function
of $z$, which cannot occur) or some $w_j(z)$ takes values in
$\mathbb{D}$ (which again forces a contradiction) or some $w_j(z)$
takes values in $\mathbb{E}$ (and this forces some other $w_r(z)$ to
assume values in $\mathbb{D}$).  Hence, $p$ has no zeros on
$\mathbb{D} \times \mathbb{T}$.

The number of zeros of $p(z,\cdot)$ (counting multiplicities) which
are contained in $\mathbb{D}$ is constant as a function of
$z \in \mathbb{D}$.  Indeed,
\[
N := \int_{\mathbb{T}} \frac{p_w(z,w)}{p(z,w)} \frac{dw}{2\pi i}
\]
is a holomorphic function on the disk which counts the roots of
$p(z,\cdot)$ that are contained in $\mathbb{D}$, by the residue
theorem.  This is valid since $p(z,w)$ is nonzero for $(z,w) \in
\mathbb{D} \times \mathbb{T}$.  Integer-valued holomorphic functions
are constant; hence, $N$ is constant.

For $z \in \mathbb{D}$, let $w_1(z), \dots, w_N(z)$ be some listing of
the $N$ roots of $p(z,\cdot)$ including multiplicities which are
contained in the disk.

From here our strategy will be to show that elementary symmetric
functions of the roots of $p(z,\cdot)$ contained in $\mathbb{D}$ are
actually rational as a function of $z \in \mathbb{D}$.  Recall the
elementary symmetric functions are given by
\[
s_k(z) := \sum_{1\leq j_1 < \dots < j_k\leq N} w_{j_1}(z) \cdots
w_{j_k}(z)
\]
These are holomorphic on the disk because the functions
\[
f_k(z) := \sum_{j=0}^N (w_j(z))^k = \int_{\mathbb{T}} \frac{p_w
  (z,w)}{p(z,w)} w^k \frac{dw}{2\pi i}
\] 
are holomorphic and the elementary symmetric functions are polynomials
in the $f_k$.  

Since the roots of $p(z,\cdot)$ tend to $\mathbb{T}$ as $z$ tends to
$\mathbb{T}$, the Schwarz reflection principle tells us that $s_N(z) =
w_1(z)\cdots w_N(z) = \prod_{j=1}^{N} w_j (z)$ can be extended
meromorphically to the extended complex plane as
$1/\overline{s_N(1/\bar{z})}$ and is therefore rational.

Consider now another family of symmetric polynomials of the roots:
\[
g_k(z) := \sum_{j=1}^{N} \prod_{t\ne j} w_t (z)^k.
\]
Again, $g_k$ is holomorphic in the disk.  We claim $f_k$ extends
meromorphically across $\mathbb{T}$ to the function
\[
\hat{f}_k(z) :=
\overline{g_k(1/\bar{z})}/\overline{s_N(1/\bar{z})^k} = \sum_{j=1}^{N}
\frac{1}{\overline{w_j(1/\bar{z})^k}}  
\]
defined and meromorphic on $\mathbb{E}\cup\{\infty\}$.  

Let $\la \in \mathbb{T}$ be a point at which $p(\la, \cdot)$ has $m$
distinct roots (only finitely many points fail to satisfy this since
$p$ and $p_w$ can have no common factor).  For $z$ in a small
neighborhood of $\la$, $p(z,\cdot)$ has $m$ distinct roots $W_1(z),
\dots, W_m(z)$ which can be given as holomorphic functions of $z$, by
the implicit function theorem.  Since $V$ is a distinguished variety,
each $W_j(z)$ is either always in the disk for $z$ in the disk or
always in $\mathbb{E}$ for $z$ in the disk.  We may assume the first
$N$ of the $W_j$ are the roots $w_1,\dots, w_N$.  Since each $W_j(z)$
tends to $\mathbb{T}$ as $z$ tends to $\mathbb{T}$, the Schwarz
reflection principle says $W_j(z) = 1/\overline{W_j(1/\bar{z})}$ for
$z$ in a neighborhood of $\la$.  Therefore, the function
$\sum_{j=0}^{N} W_j(z)^k$ is holomorphic in a neighborhood of $\la$,
agrees with $f_k(z)$ for $z \in \mathbb{D}$, and agrees with
$\hat{f}_k(z)$ on $\mathbb{E}$.  So, $f_k$ extends to be meromorphic
on the extended complex plane except at possibly finitely many points
on $\mathbb{T}$.  But, $f_k$ is bounded in a neighborhood of the
circle and its singularities are therefore removable.  Therefore,
$f_k$ is a rational function.  This implies the symmetric functions
$s_k$ are rational.

Observe that the function
\[
R(z,w) = w^N+ \sum_{j=0}^{N-1} (-1)^{N-j} s_{N-j}(z) w^j
\]
is rational and vanishes on $V$: it equals
\[
(w-w_1(z))\cdots (w-w_N(z))
\]
on the disk
and 
\[
(w-1/\overline{w_1(1/\bar{z})})\cdots (w-1/\overline{w_N(1/\bar{z})})
\]
on $\mathbb{E}$.  Since $p$ is irreducible and $N\leq m$, it must be
the case that $N=m$ and the numerator of $R$ is a constant multiple of
$p$.

As $R$ was designed to have the property that all roots of
$R(z,\cdot)$ are in $\mathbb{D}$ for $z \in \mathbb{D}$ and in
$\mathbb{E}$ for $z \in \mathbb{E}$, it follows that
\[
V' \subset \mathbb{D}^2 \cup \mathbb{T}^2 \cup \mathbb{E}^2.
\]
\end{proof}

\begin{prop}\label{symprop}
  If $V$ is a distinguished variety and $V'$ is its extension to
  $\mathbb{C}^2$ (as in Proposition \ref{compprop}) then $V'$ is
  symmetric with respect to the two-torus: for any $z\ne 0$ and $w\ne
  0$
\[
(z,w) \in V' \text{ if and only if } \left(\frac{1}{\bar{z}},
  \frac{1}{\bar{w}}\right) \in V'
\]
\end{prop}

\begin{proof} Assume $V$ is defined as the zero set of a polynomial
  $p$ of degree $(n,m)$.  As before, we may assume $V$ and $p$ are
  irreducible.  It suffices to prove $p$ is essentially $\mathbb{T}^2$-symmetric
  (recall this means $p = c\tilde{p}$ for a unimodular constant $c$),
  since then $p(z,w) = 0$ if and only if $p(1/\bar{z}, 1/\bar{w})=0$.
  Write
\[
p(z,w) = \sum_{j=0}^{m} p_j(z) w^j
\]
for one variable polynomials $p_0,p_1, \dots, p_m$ with degrees at
most $n$.  Then,
\[
\tilde{p}(z,w) = \sum_{j=0}^{m} \tilde{p}_{m-j} (z) w^j
\]
where $\tilde{p}_j(z)$ is the one variable reflection of the
polynomial $p_j$:
\[
\tilde{p}_j (z) = z^n \overline{p_j\left(\frac{1}{\bar{z}}\right)}.
\]

For any $z \in \mathbb{T}$, all zeros of $p(z,\cdot)$ are on the
circle since $V' \subset \mathbb{D}^2 \cup \mathbb{T}^2 \cup
\mathbb{E}^2$. Since
\[
|p(z,w)| = |\tilde{p}(z,w)| \text{ for all } (z,w) \in \mathbb{T}^2
\]
 $p(z,\cdot)$ and $\tilde{p}(z,\cdot)$ have the same zeros when $z \in
 \mathbb{T}$ (including repeated roots which can be measured by the
 order of vanishing).

This implies that for $z \in \mathbb{T}$
\[
\tilde{p}_0 (z) p(z,\cdot) = p_m (z) \tilde{p} (z,\cdot)
\]
since the leading coefficients of these two polynomials with the same
zeros are equal; i.e. for $(z,w) \in \mathbb{T}\times \mathbb{C}$
\begin{equation} \label{onlyhappen}
\tilde{p}_0 (z) p(z,w) = p_m (z) \tilde{p} (z,w).
\end{equation}
This can only happen if \eqref{onlyhappen} holds for all $(z,w) \in
\mathbb{C}^2$, since $\mathbb{T}^2$ is a set of uniqueness for
polynomials.

By irreducibility of $p$ we see that $p = c\tilde{p}$ for some
constant $c$.  Since $|p| = |\tilde{p}|$ on $\mathbb{T}^2$, the
constant $c$ must be unimodular.  This proves $p$ is essentially
symmetric.
\end{proof}  

\begin{proof}[Proof of Lemma \ref{corrlemma}]
Let $p$ be a polynomial of degree $(n,m)$ defining the distinguished
  variety $V$ in the sense of Definition \ref{pdefines}.  Then,
\begin{equation} \label{qandp}
q(z,w) := z^n p\left(\frac{1}{z}, w \right) = w^m \overline{p\left(\bar{z},
  \frac{1}{\bar{w}} \right)} 
\end{equation}
defines a polynomial of degree strictly equal to $(n,m)$.  (The middle
expression has degree $m$ in $w$ and the expression on the right has
degree $n$ in $z$.)

Also, by equation \eqref{qandp}, since $p$ has no zeros on the set
\[
(\mathbb{E}\times \mathbb{D})\cup (\mathbb{D}\times \mathbb{E}) \cup
(\mathbb{T}\times \mathbb{D}) \cup (\mathbb{D} \times \mathbb{T})
\]
it follows that $q$ has no zeros on $\undist$ except possibly at
$(0,0)$.  Since zeros of two variable polynomials are never isolated,
we may conclude $q$ has no zeros on $\undist$.

Finally, since $p$ is \etsym\ it is clear that $q$ is too.
\end{proof}

\section{Sum of Squares Formulas} \label{sossec}

In this section we prove Theorems \ref{thm:tildenozeros},
\ref{symstablethm} and \ref{sosdist}.  We will identify $V$ with its
extension $V'$ to $\mathbb{C}^2$ from the previous section via a
defining polynomial of minimal degree.

\begin{proof}[Proof of Theorem \ref{thm:tildenozeros}] Let $q$ be a
  \tsym\ polynomials with no zeros on $\mathbb{D}^2$ (or $\undist$).
  For each $t \in (0,1]$ define
\[
q_t(z,w) := q(tz,w)
\]
(not to be confused with the partial derivative notations).  Then,
$q_t$ has no zeros on $\mathbb{D}^2$ and therefore
\[
|q_t(z,w)|^2 - |\widetilde{q_t}(z,w)|^2 \geq 0 \text{ on } \cbidisk.
\]
Since $q$ is \tsym, it can be shown that
\[
\widetilde{(q_t)} (z,w) = t^n q(z/t, w).
\]
Considering the following expression for $(z,w) \in \cbidisk$
\[
\frac{|q(tz,w)|^2 - |t^nq(z/t,w)|^2}{1-t^2} \geq 0 
\]
and taking the limit as $t\nearrow 1$ yields
\[
n|q(z,w)|^2 - 2\text{Re}(zq_z(z,w) \overline{q(z,w)}) \geq 0 \text{ on
} \cbidisk.
\]
By the Lemma \ref{modlemma} with $a=1, b=0$,
\[
\begin{aligned}
&n^2|q(z,w)|^2 - 2\text{Re}(zq_z(z,w) \overline{n q(z,w)})\\
=& |\widetilde{q_z}(z,w)|^2 - |zq_z(z,w)|^2 \geq 0 \text{ on }
\cbidisk
\end{aligned}
\]
Therefore, any zero of $\widetilde{q_z}$ on $\cbidisk$ must also be a
zero of $z q_z$, and by the identity $nq = \widetilde{q_z} + zq_z$
(from Lemma \ref{refformulas}), any zero of $\widetilde{q_z}$ on
$\cbidisk$ is a zero of $q$.  So, if $q$ has no zeros on
$\mathbb{D}^2$ (resp. $\undist$), then $\widetilde{q_z}$ has no zeros
on $\mathbb{D}^2$ (resp. $\undist$).  Similar statements hold for
$\widetilde{q_w}$.

By Lemma \ref{modlemma} (used three times), when $a, b\geq 0$ and
$(z,w) \in \cbidisk$ we have
\begin{align} \label{nextobserve}
&|a\widetilde{q_z} + b \widetilde{q_w}|^2 - |azq + bwq_w|^2 \\ \nonumber
=& (an+bm)[ (an+bm)|q|^2 - 2\text{Re}[(azq_z+ bw q_w)\bar{q}] ] \\
\nonumber
=& (an+bm)[ (a/n)( n^2|q|^2 - 2\text{Re}(zq_z n\bar{q})) + (b/m)(
m^2|q|^2 - 2\text{Re}(wq_w m\bar{q})) ] \\ \nonumber
=& (an+bm)[ (a/n)( |\widetilde{q_z}|^2 - |zq_z|^2) + (b/m)(
|\widetilde{q_w}|^2 - |wq_w|^2 ) ] \\ \nonumber
& \geq (an+bm) [ (a/n)(1-|z|^2)|\widetilde{q_z}|^2 +
(b/m)(1-|w|^2)|\widetilde{q_w}|^2 ]
\end{align}

where the inequality follows from
\[
\begin{aligned}
&|\widetilde{q_z}|^2 - |zq_z|^2 = |\widetilde{q_z}|^2(1-|z|^2)
+|z|^2(|\widetilde{q_z}|^2 - |q_z|^2)  \\
&\geq  (1-|z|^2) |\widetilde{q_z}|^2.
\end{aligned}
\]
(Note: for brevity we are omitting the argument $(z,w)$ in front of
all of the polynomials above).

Since $\widetilde{q_z}$ and $\widetilde{q_w}$ have no zeros on
$\mathbb{D}^2$, it now follows that for $a,b > 0$
$a\widetilde{q_z}+b\widetilde{q_w}$ has no zeros on $\mathbb{D}^2$.
If $a\widetilde{q_z}+b\widetilde{q_w}$ has a zero on
$\overline{\mathbb{D}^2} \setminus \mathbb{D}^2$, then the left side
of \eqref{nextobserve} vanishes to at least order two.  This implies
$(1-|z|^2) |\widetilde{q_z}|^2$ and $(1-|w|^2) |\widetilde{q_w}|^2$
both vanish to order at least two.  Since $(1-|z|^2)$ and $(1-|w|^2)$
vanish to at most order 1 at a point in $\cbidisk\setminus
\mathbb{D}^2$, it follows that both $\widetilde{q_z}$ and
$\widetilde{q_w}$ vanish at a zero of $a\widetilde{q_z} + b
\widetilde{q_w}$.  Therefore, $a\widetilde{q_z}+b\widetilde{q_w}$ has
no zeros on $\overline{\mathbb{D}^2} \setminus (Z_{\widetilde{q_z}}
\cap Z_{\widetilde{q_w}})$ when $a,b >0$.
\end{proof}

\begin{proof}[Proof of Corollary \ref{symcor4}] There is no harm in
  assuming $q$ is \tsym.  Setting $a=m$ and $b=n$ in
  \eqref{nextobserve} yields
\[
\begin{aligned}
  & 2\text{Re}((nmq(z,w) - mzq_z(z,w) - nwq_w(z,w))\overline{q(z,w)})\\
& \geq
  (m/n)(1-|z|^2)|\widetilde{q_z}(z,w)|^2 +
  (n/m)(1-|w|^2)|\widetilde{q_w}(z,w)|^2 \geq 0
\end{aligned}
\]
on $\mathbb{D}^2$.  Since $\widetilde{q_z}$ and $\widetilde{q_w}$ have
no zeros on $\mathbb{D}^2$ (resp. $\undist$) when $q$ has no zeros on
$\mathbb{D}^2$ (resp. $\undist$), it follows that
\[
nmq(z,w) - mzq_z(z,w) - n wq_w(z,w)
\]
has no zeros on $\mathbb{D}^2$ (resp. $\undist$) when $q$ has no zeros
on $\mathbb{D}^2$ (resp. $\undist$).

By Lemma \ref{refformulas} the reflection of this
polynomial at the degree $(n,m)$ is equal to
\[
-nmq(z,w)+mzq_z(z,w)+nwq_w(z,w)
\]
and therefore 
\[
nmq(z,w)-mzq_z(z,w) - nwq_w(z,w)
\]
is \etsym.

The statement that $mzp_z - n wp_w$ defines a distinguished
variety when $p$ defines a distinguished variety follows by applying
the previous arguments to $q(z,w) = z^n p(1/z, w)$ and then converting
back to expressions involving $p$ and $p$'s partial derivatives.
\end{proof}

\begin{proof}[Proof of Theorem \ref{symstablethm}]
Apply Theorem \ref{stablethm} to $a\widetilde{q_z}+b\widetilde{q_w}$
and use the formula in Lemma \ref{modlemma} to prove the existence of
$\vec{A}$ and $\vec{B}$ satisfying
\[
\begin{aligned}
  (an+bm)|q(z,w)|^2 &- 2\text{Re}( (azq_z(z,w)+bwq_w(z,w)) \overline{q(z,w)})\\
  &= (1-|z|^2) |\vec{A}(z,w)|^2 + (1-|w|^2) |\vec{B}(z,w)|^2
\end{aligned}
\]

This can be polarized into the equation
\begin{align} \label{sospolar}
  &(an+bm) q(z,w) \overline{q(Z,W)} \\ \nonumber
  &- (azq_z(z,w)+bwq_w(z,w))\overline{q(Z,W)} \\ \nonumber 
  &- q(z,w) \overline{aZq_z(Z,W)+bWq_w(Z,W)}\\ \nonumber
  =& (1-z \bar{Z})
  \ip{\vec{A}(z,w)}{\vec{A}(Z,W)} + (1-w \bar{W})
  \ip{\vec{B}(z,w)}{\vec{B}(Z,W)}.
\end{align}

Suppose now that $q$ is a product of distinct irreducible factors.
Suppose $\vec{A}$ and $\vec{B}$ have infinitely many zeros on $Z_q$.
This means they must have a factor in common with $q$.  Specifically,
$q=fg$ for some relatively prime polynomials $f$ and $g$, and $f$
divides $\vec{A}$ and $\vec{B}$. Substituting $(Z,W) = (0,0)$ into
\eqref{sospolar} and rearranging produces
\[
\begin{aligned}
(an+bm)q(z,w) \overline{q(0,0)} &- \ip{\vec{A}(z,w)}{\vec{A}(0,0)} -
\ip{\vec{B}(z,w)}{\vec{B}(0,0)} \\
&= (azq_z(z,w)+bwq_w(z,w))\overline{q(0,0)}
\end{aligned}
\]
The polynomial $f$ divides the left hand side and hence divides
$azq_z+bwq_w$ (note $q(0,0) \ne 0$).  Using $q=fg$ we have
\[
\begin{aligned}
azq_z(z,w) + bwq_w(z,w) =& (azf_z(z,w)+bwf_w(z,w))g(z,w) \\
&+ (azg_z(z,w)
+b wg_w(z,w) ) f(z,w)
\end{aligned}
\]
and therefore $f$ divides $azf_z+bwf_w$ since $f$ and $g$ are
relatively prime.  Since the degree of $azf_z+bwf_w$ is less than or
equal to the degree of $f$ (in each variable separately) and since $f(0,0)
\ne 0$, this implies $azf_z + bwf_w \equiv 0$.  As all factors of $q$
are \etsym\ polynomials with no zeros on $\mathbb{D}^2$ (and can be
made \tsym), applying the formula in Lemma \ref{modlemma} to $f$ we
arrive at the contradictory conclusion that $f\equiv 0$.    

Thus, $\vec{A}$ and $\vec{B}$ can only have finitely many zeros on
$Z_q$ when $q$ is a product of distinct irreducible factors.

\end{proof}

\begin{proof}[Proof of Theorem \ref{sosdist}]
  By Lemma \ref{corrlemma}, Theorem \ref{symstablethm} can be applied
  to $q$ given by
\[
q(z,w) := z^n p(1/z,w).
\]
Before doing so we list some simple formulas relating partial
derivatives of $q$ and $p$:
\[
\begin{aligned}
zq_z(z,w) &= nz^n p(1/z, w) - z^{n-1} p_z (1/z,w) \\
wq_w (z,w) & = z^n p_w(1/z,w) 
\end{aligned}
\]
and therefore
\begin{align} \label{therefore}
azq_z(z,w) + bwq_w(z,w) &= anz^np(1/z,w) -(az^{n-1} p_z(1/z,w)- bwz^n
p_w(1/z,w)) \\ \nonumber
& = a z^n \widetilde{p_z}(1/z, w) + bwz^n p_w(1/z, w)
\end{align}

By Theorem \ref{symstablethm} using equation \eqref{therefore}, there
exists a $\mathbb{C}^n$-valued polynomial $\vec{A}(z,w)$ of degree at
most $(n-1,m)$ and a $\mathbb{C}^m$-valued polynomial $\vec{B}(z,w)$
of degree at most $(n,m-1)$ with finitely many common zeros on $Z_q$
such that
\[
\begin{aligned}
& (an+bm)|z^n p(1/z,w)|^2 \\
&- 2\text{Re}([anz^np(1/z,w) -(az^{n-1} p_z(1/z,w)
- bwz^n p_w(1/z,w))] \overline{z^n p(1/z,w)}) \\
&= (1-|z|^2) |\vec{A}(z,w)|^2 + (1-|w|^2) |\vec{B}(z,w)|^2
\end{aligned}
\]
Replacing $z$ with $1/z$ and multiplying through by $|z^n|^2$ yields
\[
\begin{aligned}
& (an+bm)|p(z,w)|^2 - 2\text{Re}((anp(z,w) -(az p_z(z,w)
- bw p_w(z,w))) \overline{p(z,w)}) \\
& = (bm-an)|p(z,w)|^2 + 2\text{Re}((az p_z(z,w)
- bw p_w(z,w)) \overline{p(z,w)})\\
&= (|z|^2-1) |z^{n-1} \vec{A}(1/z,w)|^2 + (1-|w|^2) |z^n \vec{B}(1/z,w)|^2
\end{aligned}
\]

Defining $\vec{P}(z,w) = z^{n-1} \vec{A}(1/z,w)$ and $\vec{Q}(z,w) =
z^n \vec{B}(1/z,w)$, we get
\[
\begin{aligned}
&  (bm-an)|p(z,w)|^2 + 2\text{Re}((az p_z(z,w)
  - bw p_w(z,w)) \overline{p(z,w)})\\
  & +(1-|z|^2) |\vec{P}(z,w)|^2 \\
  &= (1-|w|^2) |\vec{Q}(z,w)|^2
\end{aligned}
\]
By the polarization theorem for holomorphic functions, we get the
formula:
\begin{align} \label{sosdistformula}
(bm-an) p(z,w)\overline{p(Z,W)} &+ \left(azp_z(z, w)-bw
    p_w(z,w)\right)\overline{p(Z,W)}\\ \nonumber & +
  (1-z\bar{Z})\ip{\vec{P}(z,w)}{\vec{P}(Z, W)} \\
  \nonumber =&
  p(z,w)\left(b\overline{W p_w(Z,W)}-a\overline{Z p_z(Z,
      W)}\right)\\ \nonumber &+ (1-w\bar{W}) \ip{\vec{Q}(z,w)}{\vec{Q}(Z,W)}
\end{align}

For $(z,w), (Z,W) \in V$, the above reduces to
\begin{equation} \label{sosonV2}
(1-z\bar{Z}) \ip{\vec{P}(z,w)}{\vec{P}(Z,W)} = (1-w\bar{W})
\ip{\vec{Q}(z,w)}{\vec{Q}(Z,W)} 
\end{equation}
 
The fact that none of the components of $\vec{P}$ or $\vec{Q}$ is
identically zero when the degree $(n,m)$ is minimal will follow from
our proof of Theorem \ref{repthm}.  For the moment, we at least know
that $\vec{P}$ and $\vec{Q}$ have finitely many common zeros on $V$.
(There is no circular reasoning going on here as will made apparent at
the end of the proof of Theorem \ref{repthm}.)
\end{proof}

\section{Representation of Distinguished Varieties} \label{mainthmsec}

\begin{proof}[Proof of Theorem \ref{repthm}:]
  We continue directly from the end of the proof of Theorem
  \ref{sosdist}.  Equation \eqref{sosonV2} can be rearranged into
\[
\begin{aligned}
&\ip{\vec{P}(z,w)}{\vec{P}(Z,W)} + w\bar{W} \ip{\vec{Q}(z,w)}{\vec{Q}(Z,W)} \\
& = z\bar{Z} \ip{\vec{P}(z,w)}{\vec{P}(Z,W)} + \ip{\vec{Q}(z,w)}{\vec{Q}(Z,W)}
\end{aligned}
\]
and from here much of the proof follows a standard ``lurking
isometry''/systems theory argument similar to that found in
\cite{AM05}.

The map sending
\[
\begin{pmatrix} \vec{Q}(z,w) \\ z \vec{P}(z,w) \end{pmatrix}
\mapsto \begin{pmatrix} w\vec{Q}(z,w) \\ \vec{P}(z,w) \end{pmatrix}
\]
for each $(z,w) \in V$ defines a unitary on the span of elements of
$\mathbb{C}^{m+n}$ of the form on the left to the span of the elements
of $\mathbb{C}^{m+n}$ of the form on the right which we may extend to
a unitary matrix $U: \mathbb{C}^m \oplus \mathbb{C}^n \to \mathbb{C}^m
\oplus \mathbb{C}^n$.  Let us write $U$ in block form
\[
U = \begin{pmatrix} A & B \\ C & D \end{pmatrix}.
\]
Then, 
\begin{align}
A \vec{Q}(z,w) + zB\vec{P}(z,w) &= w \vec{Q}(z,w) \label{AB} \\ 
C \vec{Q}(z,w) + zD\vec{P}(z,w) &= \vec{P}(z,w) \label{CD}
\end{align}
for all $(z,w) \in V$. 

This implies
\[
\begin{pmatrix} \mathbf{0} \\ \mathbf{0} \end{pmatrix} =
\begin{pmatrix} A - wI_m & zB \\ C & zD - I_n \end{pmatrix}
\begin{pmatrix} \vec{Q}(z,w) \\ \vec{P}(z,w) \end{pmatrix}
\]
for all $(z,w) \in V$.  By Theorem \ref{symstablethm} and by
construction, $\vec{P}$ and $\vec{Q}$ have at most finitely many
common zeros on $V$.  Therefore,
\[
\det\begin{pmatrix} A - wI_m & zB \\ C & zD - I_n \end{pmatrix} = 0
\]
for all but finitely many points on $V$ and hence all points on $V$.
Since this is a polynomial of degree at most $(n,m)$, it must equal a
constant multiple of $p$, by minimality of $p$.  This proves the first
representation formula in Theorem \ref{repthm}.

Equation \eqref{CD} implies
\[
(I_n - zD)^{-1} C\vec{Q}(z,w) = \vec{P} (z,w)
\]
and coupled with equation \eqref{AB} we have
\begin{equation} \label{evenmore}
\{A+zB(I_n - zD)^{-1}C - w I_m \} \vec{Q}(z,w) = 0 \text{ for } (z,w)
\in V.
\end{equation}

Define $\Phi(z) := A+zB(I_n - zD)^{-1}C$. A well-known calculation
proves $\Phi(z)$ is a rational inner function.  Indeed, the fact that
$U$ is a unitary and that
\[
 U \begin{pmatrix} I_m \\ z(I_n-zD)^{-1}C \end{pmatrix}  =
\begin{pmatrix} \Phi(z) \\ (I_n-zD)^{-1}C \end{pmatrix}
\]
implies
\[
I_m+ |z|^2 C^* (I_n-\bar{z}D^*)^{-1} (I_n - zD)^{-1} C = \Phi(z)^*
\Phi(z) +  C^* (I_n-\bar{z}D^*)^{-1} (I_n - zD)^{-1} C.
\]
Rearranging gives
\[
I_m- \Phi(z)^* \Phi(z) = (1-|z|^2)C^* (I_n-\bar{z}D^*)^{-1} (I_n -
zD)^{-1} C.
\]

\begin{claim} $D$ has no unimodular eigenvalues.
\end{claim}
\begin{proof} Suppose there is a nonzero vector $\vec{v} \in
  \mathbb{C}^n$ such that $D\vec{v} = \la \vec{v}$ for some $\la \in
  \mathbb{T}$.  Then,
\[
U \begin{pmatrix} \mathbf{0} \\ \vec{v} \end{pmatrix}
= \begin{pmatrix} B\vec{v} \\ \la \vec{v} \end{pmatrix}
\]
and the unitarity of $U$ implies $B\vec{v} = 0$.  Hence,
$\begin{pmatrix} \mathbf{0} \\ \vec{v} \end{pmatrix}$ is an
eigenvector for $U$ and after a unitary change of coordinates (which
will not affect the $\mathbb{C}^m$ portion of $\mathbb{C}^m \oplus
\mathbb{C}^n$), $U$ can be put into the form
\[
\begin{pmatrix} A & B' & \mathbf{0} \\ C' & D' & \mathbf{0} \\
  \mathbf{0}^t & \mathbf{0}^t & \la \end{pmatrix}.
\]

This implies $p$ is a constant multiple of
\[
(\la z-1) \det\begin{pmatrix} A-wI_m & zB' \\ C' & zD'-I_{n-1} \end{pmatrix}
\]
which contradicts the fact that $p$ defines a distinguished variety.
\end{proof}

By the claim, it now follows that $\Phi(z)^* \Phi(z) = I_m$ for $z \in
\mathbb{T}$. 
 In words, $\Phi$ is unitary valued on the circle.

On $V\cap \mathbb{D}^2$, zeros of $\vec{Q}$ coincide with zeros of
$\vec{P}$ by equation \eqref{sosonV2}. Hence, we see that $\vec{Q}$ has
at most finitely many zeros on $V$.  So, by \eqref{evenmore}
\[
\det(wI_m - \Phi(z)) = 0
\]
for all $(z,w) \in V$. The rational function
\[
r(z,w) = \det(wI_m - \Phi(z))
\]
has numerator with degree at most $m$ in $w$ and it vanishes on $V$,
and this implies that it must be a constant multiple of $p$.  As was
already mentioned above, 
\[
\Phi(z) \vec{Q}(z,w) = w \vec{Q}(z,w) \text{ for } (z,w) \in V
\]
and this completes the proof of Theorem \ref{repthm}.

If any of the components of $\vec{P}$ or $\vec{Q}$ had been
identically zero this same proof could have yielded a polynomial of
strictly lower degree than $p$ (in either $z$ or $w$) which vanished
on $V$.  This cannot happen since $p$ is by definition minimal.
Therefore, we have also completed the proof Theorem \ref{sosdist}.
\end{proof}

\begin{proof}[Proof of Theorem \ref{refinerep}]
The existence of $\vec{Q}$ as stated in the theorem follows from the
previous proof.  

If $V=Z_p$ has no singularities on $\mathbb{T}^2$, then $p_z$ and $p_w$
have no common zeros on $V\cap \mathbb{T}^2$.  Using the formulas $np =
zp_z+\widetilde{p_z}$ and $mp = wp_w + \widetilde{p_w}$, all zeros of
$p_z$ and $p_w$ on $\mathbb{T}^2$ must occur on $V$.  Hence, $p_z$ and
$p_w$ have no common zeros on all of $\mathbb{T}^2$.  Reverting to
\tsym\ polynomials with no zeros on $\undist$:
\[
q(z,w) := z^n p(1/z, w)
\]
we see that $q_z$ and $q_w$ have no common zeros on $\mathbb{T}^2$.
By Theorem \ref{thm:tildenozeros}, $\widetilde{q_z}$ and $\widetilde{q_w}$
each have no zeros on $\undist$.  So, again by Theorem
\ref{thm:tildenozeros}, $\widetilde{q_z}+\widetilde{q_w}$ has no zeros
on $\cbidisk \setminus (Z_{\widetilde{q_z}} \cap Z_{\widetilde{q_w}}) =
  \cbidisk$.  

Theorem \ref{GWthm} can be applied to
$\widetilde{q_z}+\widetilde{q_w}$ and the proof of Theorem
\ref{symstablethm} can be repeated with the additional knowledge that
$\vec{B}(z,w)$ when written in matrix form as
\[
\vec{B}(z,w) = B(z) \begin{pmatrix} 1 \\ w \\ \vdots \\
  w^{m-1} \end{pmatrix} 
\]
has the property that $z^n\overline{B(1/\bar{z})}$ is invertible for
$z \in \overline{\mathbb{D}}$.  Repeating the proof of Theorem
\ref{sosdist}, we see that
\[
\vec{Q}(z,w) = z^n w^{m-1} \overline{\vec{B}(1/\bar{z}, 1/\bar{w})} =
z^n \overline{B(1/\bar{z})} \begin{pmatrix} w^{m-1} \\ w^{m-2} \\
  \vdots \\ 1 \end{pmatrix} 
\]
and therefore the matrix $Q(z)$ satisfying 
\[
\vec{Q}(z,w) = Q(z)  \begin{pmatrix} 1 \\ w \\ \vdots \\
  w^{m-1} \end{pmatrix} 
\]
is equal to $z^n \overline{B(1/\bar{z})}$ with its columns in reverse
order.  Thus, $Q(z)$ is invertible for $z\in
\overline{\mathbb{D}}$.

\end{proof}

\begin{proof}[Proof of Theorem \ref{extendthm}]
  If $f \in \mathbb{C}[z,w]$, then for any vectors $\vec{v_1}$,
  $\vec{v_2} \in \mathbb{C}^m$ we have that for each $z \in
  \mathbb{D}$
\[
\begin{aligned}
|\ip{f(zI_m,\Phi(z)) \vec{v_1}}{\vec{v_2}}| &\leq \sup_{\lambda \in
  \mathbb{T}} |\ip{f(\lambda I_m, \Phi(\lambda))\vec{v_1}}{\vec{v_2}}|
\\
&\leq \sup_{V\cap\mathbb{D}^2} |f| |\vec{v_1}| |\vec{v_2}|
\end{aligned}
\]
by the maximum modulus principle and since $f(\lambda I_m,
\Phi(\lambda))$ is normal with eigenvalues given by the values of $f$
on $V$ when $\lambda \in \mathbb{T}$ (recall $\Phi$ is unitary on the
circle).  Therefore, in operator norm
\[
||f(zI_m, \Phi(z))|| \leq \sup_{V\cap\mathbb{D}^2} |f|.
\]

Now, we examine 
\[
F(z,w) := (1,0,\dots, 0) Q(z)^{-1} f(zI_m, \Phi(z)) \vec{Q}(z,w).
\]
For $(z,w) \in V\cap\mathbb{D}^2$, $f(zI_m, \Phi(z)) \vec{Q}(z,w) =
f(z,w) \vec{Q}(z,w)$ and therefore when $(z,w) \in V\cap\mathbb{D}^2$
\[
F(z,w) = f(z,w) (1,0,\dots, 0) Q(z)^{-1} Q(z) \begin{pmatrix} 1 \\ w
  \\ \vdots \\ w^{m-1} \end{pmatrix} = f(z,w)
\]
i.e. $F$ is an extension of $f$.  

The estimates on $F$ are now straightforward:
\[
\begin{aligned}
|F(z,w)| &\leq ||Q(z)^{-1}||\ ||f(z I_m, \Phi(z))||\ |\vec{Q}(z,w)|
\\
&\leq ||Q(z)^{-1}||\ |\vec{Q}(z,w)| \sup_{V\cap \mathbb{D}^2} |f| \\ 
& \leq \sqrt{m} ||Q(z)^{-1}||\ ||Q(z)|| \sup_{V\cap \mathbb{D}^2} |f|
\end{aligned}
\]
for all $(z,w) \in \mathbb{D}^2$ since 
\[
|\vec{Q}(z,w)| = \left|Q(z)  \begin{pmatrix} 1 \\ w \\ \vdots \\
  w^{m-1} \end{pmatrix} \right| \leq \sqrt{m} ||Q(z)||
\]
for $|w|\leq 1$.
\end{proof}

\section{Appendix} \label{appendix} In this appendix we explain how to
obtain Theorem \ref{GWthm} from the results in Knese \cite{gK08}.  Let
$q$ be a polynomial of degree at most $(n,m)$ with no zeros on
$\cbidisk$.

Define a probability measure $\rho$ on $\mathbb{T}^2$ by
\[
d\rho = \frac{c^2}{|q(z, w)|^2} \frac{dz}{2\pi i} \frac{dw}{2\pi i}
\]
where $c$ is chosen to make $\rho$ a bona fide probability
measure. Let $\ip{\cdot}{\cdot}_{\rho}$ denote the standard inner
product on $L^2(\rho)$.  Consider the following $n$ dimensional
subspace of 2 variable polynomials defined using the given inner
product:
\[
\begin{aligned}
S_1 = &\{\text{polynomials of degree at most } (n-1,m)\} \\
&\ominus_\rho
\{wp(z,w): p \text{ is a polynomial of degree at most } (n-1,m-1)\}
\end{aligned}
\]
where we have written $\ominus_\rho$ to emphasize that this is an
orthogonal complement performed using $\ip{}{}_{\rho}$.  

\begin{claim} No nonzero element of $S_1$ is divisible by a polynomial
  of the form $L(z,w) = w-w_0$ where $w_0 \in \overline{\mathbb{D}}$.  
\end{claim}
\begin{proof} Suppose $L$ divides a nonzero $p \in S_1$; i.e. $p(z,w) =
  (w-w_0)r(z,w)$ for some polynomial $r$ of degree at most
  $(n-1,m-1)$.  Then,   
\[
\begin{aligned}
|w_0|^2||r||^2_{L^2(\rho)} &= ||w_0 r||^2_{L^2(\rho)} =
||p-wr||^2_{L^2(\rho)} \\
&= ||p||^2_{L^2(\rho)} + ||r||^2_{L^2(\rho)}
\end{aligned}
\]
since $p$ is orthogonal to $wr$ and since multiplication by $w$ is an
isometry on $L^2(\rho)$. (We are slightly abusing notation here and
confounding $w$ with the function $(z,w) \mapsto w$.) Therefore,
$||p||^2_{L^2(\rho)} = (|w_0|^2-1)||r||^2_{L^2(\rho)}$ which can only
be positive when $|w_0| > 1$.
\end{proof}

Let $KS_1((z,w),(Z,W))$ be the reproducing kernel for $S_1$ using the
given inner product $\ip{}{}_\rho$ (for details on reproducing kernels
in this setting see \cite{gK08}).  For any orthonormal basis $\{E_1,
\dots, E_n\}$ of $S_1$ it is a fact that
\[
KS_1((z,w),(Z,W)) = \sum_{j=1}^{n} E_j(z,w) \overline{E_j(Z,W)}
\]
or in vector polynomial notation with $\vec{E}(z,w) = (E_1(z,w)
,\dots, E_n(z,w))^t$ we have
\[
KS_1((z,w),(Z,W)) = \ip{\vec{E}(z,w)}{\vec{E}(Z,W)}.
\]
Writing $\vec{E}(z,w)$ in matrix form
\begin{equation} \label{matrixform}
\vec{E}(z,w) = E(w) \begin{pmatrix} 1 \\ z \\ \vdots \\
  z^{n-1} \end{pmatrix} 
\end{equation}
we make the following claim.

\begin{claim} The $n\times n$ matrix $E(w)$ is invertible for all $w
  \in \overline{\mathbb{D}}$.  
\end{claim}
\begin{proof} Suppose $E(w)$ is singular for some $w_0 \in
  \overline{\mathbb{D}}$.  Then, there is a nonzero vector $\vec{v} \in
  \mathbb{C}^n$ such that $\vec{v}^t E(w_0) = \mathbf{0}^t$.  This
  implies $\vec{v}^t \vec{E}(z,w) \in S_1$ vanishes on the set $w = w_0$.  By
  the previous claim this can only occur if $\vec{v}^t\vec{E}$ is
  identically zero, which cannot happen since $\{E_1, \dots, E_n\}$ is
  a basis for $S_1$ and $\vec{v}$ is nonzero.
\end{proof}
  
Consider a second subspace 
\[
\begin{aligned}
\widetilde{S_2} := &\{\text{polynomials of degree } (n,m-1)\} \\
& \ominus_{\rho} \{\text{polynomials of degree } (n-1,m-1)\}.
\end{aligned}
\]
Using arguments similar to the above it can be shown that the
reproducing kernel $K\widetilde{S_2}$ can be written as
\[
K\widetilde{S_2} ((z,w),(Z,W)) = \ip{\vec{F}(z,w)}{\vec{F}(Z,W)}
\]
where $\vec{F}$ is a $\mathbb{C}^m$-valued polynomial which when
written in matrix form
\[
\vec{F}(z,w) = F(z) \begin{pmatrix} 1 \\ w \\ \vdots \\
  w^{m-1} \end{pmatrix} 
\]
has the property that $z^n \overline{F(1/\bar{z})}$  is invertible for
all $z \in \overline{\mathbb{D}}$.

Finally we can give the connection to Theorem \ref{GWthm}.  Theorems
4.5 and 5.1 in Knese \cite{gK08} say
\begin{equation} \label{theoremssay}
\frac{|q(z,w)|^2}{c^2} - \frac{|\tilde{q}(z,w)|^2}{c^2} = (1-|z|^2)
|\vec{E}(z,w)|^2 + (1-|w|^2)|\vec{F}(z,w)|^2
\end{equation}
which easily implies Theorem \ref{GWthm} by the above discussion.

\bibliography{repdv}

\end{document}